\documentclass[a4paper]{article}

\usepackage[english]{babel}
\usepackage[utf8]{inputenc}
\usepackage{amsmath}
\usepackage{amsfonts}
\usepackage{amsthm}
\usepackage{graphicx}
\usepackage[numbers]{natbib}
\usepackage{color}

\begin{document}
\newtheorem{defn}{Definition}
\newtheorem{lem}{Lemma}
\newtheorem{cor}{Corollary}
\newtheorem{thm}{Theorem}
\newtheorem{cex}{Counterexample}
\newtheorem{rem}{Remark}

\title{Sufficient Conditions for Graphicality of Bidegree Sequences}
\author{David Burstein\footnotemark[2] \footnotemark[3] \and Jonathan Rubin\footnotemark[2] \footnotemark[3]}
\date{\today}

\baselineskip 1.5em
\maketitle
\renewcommand{\thefootnote}{\fnsymbol{footnote}}
\footnotetext[2]{301 Thackeray Hall.  Department of Mathematics. University of Pittsburgh, PA 15260}
\footnotetext[3]{This work was partially supported by NSF Award DMS 1312508.}
\bibliographystyle{plain}
\pagestyle{myheadings}
\thispagestyle{plain}
\markboth{D. Burstein and J. Rubin}{Sufficient Conditions for Graphicality}
\begin{abstract}
There are a variety of existing conditions for a degree sequence to be graphic.  When a degree sequence satisfies any of these conditions, there exists a graph that realizes the sequence.  We formulate several novel sufficient graphicality criteria that depend on the number of elements in the  sequence, corresponding to the number of nodes in an associated graph, and the mean degree of the sequence.  These conditions, which are stated in terms of bidegree sequences for directed graphs, are easier to apply than classic necessary and sufficient graphicality conditions involving multiple inequalities.  They are also more flexible than more recent graphicality conditions, in that they imply graphicality of some degree sequences not covered by those conditions.  The form of our results will allow them to be easily used for the generation of graphs with particular degree sequences for applications.
\end{abstract}

\indent Gale-Ryser theorem 
\newline\newline
\indent \textbf{AMS:} 05C20, 05C80, 05C82 

\section{Introduction}
Generating random graphs with various properties is relevant for a wide variety of applications, from modeling neural networks \cite{zhao11} to internet security \cite{Alb00}.  To generate an undirected random graph with a fixed number of nodes, it is natural to first select a degree distribution  through some process and then to connect the nodes in a way that is consistent with the selected distribution; similarly, a bidegree distribution would be selected if  a directed graph  were desired.  A well known issue with this procedure is that not all degree distributions are graphic; that is, it is easy to write down a sequence of $n$ natural numbers $\{ d_i \}$ such that there is no graph with $n$ nodes for which the degree of the $i$th node is $d_i$ for all $i$.  The aim of this work is to rigorously establish novel, relatively inclusive, easily checked conditions on a bidegree sequence that ensure that it is graphic and hence corresponds to one or more directed graphs. Such conditions can be used as constraints on a degree distribution to ensure that sampling from that distribution will yield a graphic degree sequence or to ease the process of verifying that a (randomly generated) bidegree sequence corresponds to a directed graph.

To start, we briefly review the background literature on sufficient conditions to guarantee graphicality, starting with some standard definitions and theorems.  In doing so, and in the rest of the paper, we will employ what is known as Hoare-Ramshaw notation for closed sets of integers, namely $[a..b] := \{ x \in \mathbb{Z}: a \leq x \leq b \}$ for $a, b \in \mathbb{Z}$.  We will also define $\mathbb{N}_0 = \mathbb{N} \cup \{0\}$ and $\mathbb{N}_0^{(n,2)} = \{ (\vec{a}, \vec{b}) : \vec{a} \in \mathbb{N}_0 \; \mbox{and} \; \vec{b} \in \mathbb{N}_0 \}$. 

\begin{defn} \label{def:graphic} A bidegree sequence $\vec{d}=(\vec{a},\vec{b})\in\mathbb{N}_0^{(n,2)}$ is graphic if there is a 0-1 binary matrix with 0's on the main diagonal such that the sum of the ith row is $a_{i}$ and ith column is $b_{i}$ for all $i = [1..n]$.  We say a bidegree sequence $\vec{d}\in\mathbb{N}_0^{(n,2)}$ is graphic with loops if there is a 0-1 binary matrix such that the sum of the ith row is $a_{i}$ and ith column is $b_{i}$.  We call $\vec{a}$ our in-degree sequence and $\vec{b}$ our out-degree sequence.  
\end{defn}

\textcolor{black}{Note that when it exists, the 0-1 binary matrix in Definition \ref{def:graphic} arises naturally as the adjacency matrix for the digraph with degrees given by $\vec{a},\vec{b}$. In this matrix, the $(i,j)$ element is 1 if the digraph includes an edge from node $j$ to node $i$ and a 0 if it does not.  To distinguish graphicality for digraphs from that for graphs, one might refer to the statement of Definition \ref{def:graphic} as defining what it means for $\vec{d}$ to be digraphic.  For simplicity, we shorten this to graphic since we do not  focus on undirected graphs in this paper.}

According to a classic theorem, we can verify the graphicality of a bidegree sequence by checking $n$ inequalities.  
{\color{black}Throughout the rest of the paper, we assume that $\sum_{i=1}^{n}a_{i}=\sum_{i=1}^{n}b_{i}$, based on the necessity of this equality for graphicality.}

\begin{thm} 
\label{thm:GR}
(Gale-Ryser/Fulkerson \cite{Gale57,Ryser57,Fulkerson60,Chen66})  Consider a bidegree sequence $\vec{d}=(\vec{a},\vec{b})$ where the $a_{i}$ are nonincreasing.   $\vec{d}$ is graphic with loops if and only if 
\begin{equation}
\label{eq:equalsum}
\sum_{i=1}^{n}a_{i}=\sum_{i=1}^{n}b_{i}
\end{equation}
 and for all $j\in [1..n-1]$,
$$\sum_{i=1}^{n}\min(b_{i},j)\geq \sum_{i=1}^{j}a_{i}.$$ Similarly, $\vec{d}$ is graphic if and only if (\ref{eq:equalsum}) holds and $\forall j\in [1..n-1]$, $$\newline\sum_{i=1}^{j}\min(b_{i},j-1)+\sum_{i=j+1}^{n}\min(b_{i},j)\geq \sum_{i=1}^{j}a_{i}.$$
\end{thm}

\textcolor{black}{The final part of the theorem cited above is in fact a revision of the classical Gale-Ryser/Fulkerson Theorem due to Antsee \cite{Antsee82}, whereas the original had required a stronger ordering in the degree sequence \cite{Chen66}.  More recent ammendations of the  Gale-Ryser/Fulkerson Theorem do exist and can be found, for example, in Berger \cite{Berger11} and Miller \cite{Miller13}.}  Miller capitalizes on the discrete ``concavity" in $j$ of the functions on the left and right hand sides of the Gale-Ryser inequalities to derive the stronger result.  Analogously, we will also exploit the ``concavity" in the inequalities to construct improved sufficient conditions for graphicality.

{\color{black}We have two motivations for constructing novel sufficient conditions for graphicality.  First, determining whether a bidegree sequence is indeed graphic from the $n$ inequalities in Theorem \ref{thm:GR} is conceptually cumbersome.  Inspection of a given degree sequence provides little intuition as to whether it is possible to construct a graph that realizes that degree sequence.  Second, verifying the $n$ inequalities in Theorem \ref{thm:GR} directly can also be computationally inefficient.  Suppose we want to construct an ensemble of networks with different bidegree sequences for the purpose of evaluating the impact of properties of the network (the bidegree sequence) on the dynamics of the system as in \cite{zhao11}.  Naturally, to complete such an analysis, we would perform a parameter sweep over various statistics of the bidegree sequence (such as the maximum, average, and minimum degree).   While linear time algorithms exist for such computations \cite{Hakimi65,Kleitman73,Erdos10}, if we knew the maximum (average degree, minimum) of our degree sequence, we would have an $O(1)$ check for graphicality as opposed to an $O(n)$ check.  Since we would want to sample {\em many} different bidegree sequences, using a linear time check would be inefficient.     For example, generating a large graph using the methods adopted by Kim et al. \cite{Kim12} requires taking a node from the graph and identifying all wirings of its outward edges that can lead to a digraph without multi-edges.  To do so, one must check graphicality of the residual bidegree sequence many times; avoiding this step by utilizing bounds on degrees that ensure graphicality could help speed up the run time of the code.}

We therefore aim to construct a theoretical result that guarantees graphicality based on various easily computable attributes (the mean, the minimum) of a bidegree sequence. 
The results of several past works  \cite{Z92,Alon10,Cairns14}  provide the following sufficient condition for graphicality in terms of the maximum and minimum values in a bidegree sequence, where $\lfloor x \rfloor$ is defined as the integer floor of $x$. 

\begin{thm}(Zverovich and Zverovich, Alon et al., and Cairns et al.) Consider a bidegree sequence $\vec{d}=(\vec{a},\vec{b})\in\mathbb{N}_0^{(n,2)}$ where $\vec{a}=\vec{b}$, $m=\min\vec{d}$ and $M=\max\vec{d}$.  If $\left \lfloor \frac{(m+M)}{4}^{2}\right\rfloor\leq mn$, then $\vec{d}$ is graphic \color{black}{with loops}.  
\end{thm}

{\color{black}The above theorem is helpful in the sense that it provides a simple criterion for determining whether there exists a digraph corresponding to a given bidegree sequence.  Unfortunately, there are in fact bidegree sequences (where $\sum a_{i} = \sum b_{i}$) that do not satisfy the conditions and are still graphic, including of course graphic sequences with $\vec{a} \neq \vec{b}$. } 

Historically, the above theorem was relevant to the problem of showing that the likelihood that a degree sequence for an undirected graph can produce a graph vanishes as $n\rightarrow\infty$ (and an analogous results holds with respect to directed graphs for a bidegree sequence that has equal in- and out-degree sums and is otherwise unconstrained).  The constraint on the maximum of the bidegree sequence in the above theorem suggested that the probability of graphicality would approach zero in this limit, since excessive growth of $M$ (e.g., proportional to $n$) with increasing $n$ would  violate the graphicality {\color{black} condition \cite{Erdos93,Barnes95,Barnes97,Nolan98}.   Ultimately, a result from Pittel \cite{Pittel99} provides a proof of this asymptotic result.  

Overall, identifying graphic degree sequences is a nontrivial problem and constructing improved sufficient conditions
for graphicality can help ease this difficulty.}

By incorporating an additional quantity, the mean number of edges of the nodes in a graph, we can prove a \textcolor{black}{refinement of Theorem 2 (Theorem 5 below), which is a sharp refinement even when $\vec{a}=\vec{b}$.}
Before doing so, however, we state and prove some intermediate results that help us build up to the ultimate results in the paper and that are already stronger than Theorem 2 under certain conditions.

{\color{black}As a final note, all of the results in this paper immediately extend to graphicality results for bipartite graphs, since every bipartite graph
can be represented as a $0-1$ rectangular binary matrix.  We can extend any rectangular binary matrix  as a square binary matrix by adding rows (or columns) 
of $0's$.  Since there is a one-to-one correspondence between digraphs (with loops) and square $0-1$ binary matrices, any sufficient conditions
that guarantee existence for digraphs carry over to bipartite graphs as well.}

 \section{Theoretical Results}
 
 To start, we prove the following Theorem, which considers the maximum of the in-degree and the maximum of the out-degree as two separate parameters.
\begin{thm} \label{thm:basic} Consider a bidegree sequence $\vec{d}=(\vec{a},\vec{b})\in\mathbb{N}_0^{(n,2)}$ {\color{black}where the entries of $\vec{a}$ are arranged in non-increasing order and assume that $\sum_{i=1}^{n}{a_{i}}=\sum_{i=1}^{n}{b_{i}}:=n\bar{c}$ where $\bar{c}$ is the average degree.}  
If  $\max a_{i}= M_{a}$ and  $\max b_{i}= M_{b},$ where \textcolor{black}{$M_{a}M_{b}\leq n\bar{c}+1$}, then $\vec{d}$ is graphic with loops.  In particular, in the special case where $M_{a}=M_{b}$, if \textcolor{black}{$\max \vec{d}\leq \left\lfloor \sqrt{n\bar{c}+1}\right\rfloor$}, then $\vec{d}$ is graphic with loops.
\end{thm}
We will prepare for the proof of the theorem with certain preliminary results.
\textcolor{black}{Before doing so, we want to point out that that the adjustment of the bounds needed to ensure graphicality (Theorem 4) rather than graphicality with loops (Theorem 3) is quite small.}  This should not be surprising as graphicality requires that the adjacency matrix have 0's on the main diagonal.  Since this restriction only affects $n$ of the $n^{2}$ entries in our adjacency matrix, it should have negligible impact in the limit of large $n$.  This concept appears again later in the paper in extending Theorem 5 to Theorem 6.  Thus, in both instances, after we prove a sufficient condition to ensure graphicality with loops, we will make a slight alteration to our sufficient condition and show that the new version guarantees the (slightly) stronger condition of graphicality.

Now, in the sufficient criteria in Theorem 2, for simplicity suppose that $\newline \left\lfloor \frac{(m+M)}{4}^{2}\right\rfloor=\frac{(m+M)}{4}^{2}$, such that $\frac{(m+M)}{4}^{2}\leq mn$ implies $M\leq \sqrt{4mn}-m\leq \sqrt{4mn}$.  We conclude that if $\bar{c}> 4m$, then Theorem 3 (with $M_{a}=M_{b}$), provides a more flexible criterion for graphicality than that given by Theorem 2. 

We also wish to differentiate Theorem 3 from the constraint provided by Chung and Lu \cite{Chung02}.   In their work, the probability of having an outgoing edge from node $j$ to node $i$ is given by a Bernoulli random variable $p_{ij}$, independent across choices of $i,j$, such that $p_{ij}=\frac{a_{i}b_{j}}{n\bar{c}}$ where $a_{i}$ is the in-degree of node $i$ and $b_{j}$ is the out-degree of node $j$.  Consequently, they require that $M_{a}M_{b}\leq n\bar{c}$ in order to ensure that the probabilities do not exceed $1$.  It is not at all obvious that this bound should translate into a sufficient condition for graphicality, and it can in fact be awkward for the Chung-Lu algorithm.   
Specifically, if $M_{a}M_{b}= n\bar{c}$, and there exists a node $i$ such that $a_{i}=M_{a}$, and a node $j$ such that $b_{j}=M_{b}$, then according to the Chung-Lu algorithm, the probability of constructing an edge between node $i$ and node $j$ is 1, which is not a natural choice  \cite{Squartini11}.

To begin the analysis, 
consider all bidegree sequences in $\mathbb{N}_0^{(n,2)}$ with maximum in-degree $M_{a}$, with maximum out-degree $M_{b}$, and with average degree $\bar{c}$, such that $n\bar{c}$ is the sum of the \textcolor{black}{in-degrees} and also the sum of the \textcolor{black}{out-degrees}.  To prove Theorem 3, we want to construct the worst possible scenario; that is, we want to identify the in-degree vector that for {\it each and every} $j$ maximizes $\sum_{i=1}^{j}a_{j}$,  and the out-degree vector that for {\it each and every} $j\in [1..n-1]$, minimizes  $\bold{F}(j,\vec{b}) := \sum_{i=1}^{n}\min(b_{i},j)$.  Once we verify that the $n$ inequalities still hold under this worst case scenario, we have consequently proved the theorem.  Since identifying the minimizer of $\bold{F}(j,\vec{b})$ is rather technical, we prove the result in the following Lemma and Corollary for clarity; Lemma 1 also follows from Lemma 2.3 in \cite{Miller13} with $a_k = \Psi(k)-\Phi(k)$ as defined below. Notice, however, that $\bold{F}(j,\vec{b})$ in Corollary \ref{cor:minimizer} is {\em not} defined as  $\sum_{i=1}^{n}\min(b_{i},j)$ and {\color{black} for completeness} we will show later in the proof of Theorem 3 that indeed 
\begin{equation}
\label{eq:Fdef}
\sum_{i=1}^{n}\min(b_{i},j)=\sum_{i=1}^{j}\#(b_{z}:b_{z}\geq i, 1\leq z\leq n).  
\end{equation}
\begin{lem}Let $\Phi:\mathbb{N}\rightarrow\mathbb{N}$, $\Psi:\mathbb{N}\rightarrow\mathbb{N}$, where $\Psi$ is a concave function; that is, $\nabla \Psi(j)=\Psi(j)-\Psi(j-1)$ is non-increasing in j.  {\color{black}Let $\gamma \in \mathbb{N}$.}  If $\nabla \Phi(j)=\Phi(j)-\Phi(j-1)=\gamma$ or $\nabla \Phi(j)=\gamma -1$ for all $j \in [\alpha..\beta]$, $\Phi(\alpha)\leq \Psi(\alpha)$ and $\Phi(\beta)\leq \Psi(\beta)$, then $\Phi(j)\leq \Psi(j)$,  for all $j\in [\alpha+1..\beta-1]$.
\end{lem}
\begin{proof}
Suppose that there exists a first contradiction such that $\Phi(k)>\Psi(k)$, for some $k$.  This implies that $\nabla \Phi(k)>\nabla \Psi(k)$ as $\Phi(k-1)\leq \Psi(k-1)$.  But since by assumption and concavity, $\nabla \Psi(j)\leq \nabla \Psi(k)\leq \nabla \Phi(k)-1\leq \min(\nabla \Phi(j))$ for all $j>k$, this implies that $ \Phi(j)>\Psi(j)$ for all $j>k$.  Since we assumed that $\Phi(\beta)\leq \Psi(\beta)$, we have arrived at a contradiction.
\end{proof}

\begin{cor}\textcolor{black}{For $\vec{b} \in \mathbb{N}_0^n$, let} $\bold{F}(j,\vec{b})=\sum_{i=1}^{j}\#(b_{z}:b_{z}\geq i, 1\leq z\leq n)$.  \textcolor{black}{Fix  $M \in \mathbb{N}$ and define the set $B_M$ of out-degree vectors as $B_M := \{ \vec{b} \in \mathbb{N}_0^n : \sum_{i=1}^{n}b_{i}=n\bar{c}$,  $\max_{i}b_{i}\leq M$, and $M \leq n\bar{c} \}$.}  Choose  $k\in\mathbb{N}$ with $k\leq n$ such that $kM\leq n\bar{c}$ and $(k+1)M>n\bar{c}$.  Define $\vec{b}^{*}$ as $b_{1}^*=\ldots=b_{k}^*=M$, $b_{k+1}^*=n\bar{c}-kM$ and $b_{l}^*=0$ for all $l>k+1$. Then under these assumptions, for every $\vec{b}\in B_M$, $\bold{F}(j,\vec{b}^{*})\leq \bold{F}(j,\vec{b})$ for each and every $j\in [1..n]$. 

\label{cor:minimizer}
\end{cor}
\begin{proof}
\textcolor{black}{Fix $M \in \mathbb{N}$.  Note that $\bold{F}(j,\vec{b})=\sum_{i=1}^{j}\#(b_{z}:b_{z}\geq i, 1\leq z\leq n)$ is concave in $j$ and $\bold{F}(M,\vec{b})=n\bar{c}$  for all $\vec{b} \in B_M$. For $\vec{b}^{*}$ as defined in the statement of the Corollary, it follows that $\bold{F}(1,\vec{b}^*) \leq \bold{F}(1,\vec{b})$ for all $\vec{b} \in B_M$. {\color{black}Note that there at most $k+1$ positive entries in the out-degree sequence $\vec{b}^*$ and $k$ of them are identical.  Consequently, for $j\leq M$, there are $k$ or $k+1$ entries in $\vec{b}^*$ with entries that equal or exceed $j$ and by definition $\nabla \bold{F}(j,\vec{b}^*)=\#(b_{z}:b_{z}\geq j, 1\leq z\leq n) \in \{ k,k+1\} $}.   Hence, applying Lemma 1 yields the desired result.}
\end{proof}

At this juncture, we now can prove \textbf{Theorem \ref{thm:basic}}.  

\begin{proof}   Consider a bidegree sequence $\vec{d}=(\vec{a},\vec{b})\in\mathbb{N}_0^{(n,2)}$ satisfying the assumptions of Theorem \ref{thm:basic}.  Let us use the out-degrees to construct an $n$ by $n$ matrix consisting only of zeros and ones (a Ferrers diagram).  For the $k$th column, starting with the first row, we write down a 1.  We continue writing 1's until the column sums to $b_{k}$ and let the remaining entries in the column be zero.  Denote the $k$th row sum as $Q_{k}$.  It follows algebraically that  $\sum_{i=1}^{n}\min(b_{i},j)=\sum_{i=1}^{j}\sum_{k=1}^{n}\bold{1}_{(b_{k}\geq i)}=\sum_{i=1}^{j}Q_{i}=\sum_{i=1}^{j}\#(b_{z}:b_{z}\geq i, 1\leq z\leq n)$.  

\textcolor{black}{Graphicality is trivial if  $M_a=1$ as long as $\sum_i a_i = \sum_i b_i$.  For $M_a>1$, we have proven that the  minimizer has out-degree sequence $\vec{b}^*$ such that $b^*_{1}=....=b^*_{M_{a}-1}=M_{b}$ (as $M_{a}M_{b}\leq n\bar{c}+1$).   Now, if $M_a M_b \leq n \bar{c}$, then $b^*_{M_a} = M_b$ as well.  On the other hand, if $M_a M_b = n \bar{c}+1$, then $b^*_{M_a} = n \bar{c} - (M_a-1)M_b = M_b-1$.  Hence, we are assured that $b^*_{M_a} \geq M_b-1$.}

\textcolor{black}{Consequently, for $j\leq M_{b}-1$, $ \sum_{i=1}^{j}a_{i}\leq jM_{a}$, as $\max a_{i}\leq M_{a}$, and furthermore, $jM_{a} = \sum_{i=1}^j \#(b_k^* \geq i) \leq \sum_{i=1}^{n}\min(b_{i},j)$.  Meanwhile, for all $j\geq M_{b}$,  $\sum_{i=1}^{n}\min(b_{i},j)=n\bar{c}$. Hence, by  Theorem \ref{thm:GR}, the result is proved. }  

\end{proof}

{\color{black}The sufficient condition in Theorem \ref{thm:basic} is the best we can do without knowing more information regarding our degree sequence, as 
illustrated in the following counterexample.}

\begin{cex}{\color{black}There exists a degree sequence $\vec{d}$ such that $\sum a_{i} = \sum b_{i} = n\bar{c}$ , $M_{a}M_{b} = n\bar{c} + 2,$
and $\vec{d}$ is not graphic with loops.}
\end{cex}
\begin{proof}
{\color{black}Consider a degree sequence with $M_a \geq 2, M_b>2$ where $b_{1}=...=b_{M_{a}-1}=M_{b}$ and $b_{M_{a}} = n\bar{c} - (M_{a}-1)M_{b} = M_{b} - 2$.
Furthermore let $a_{1}=...=a_{M_{b}-1}=M_{a}$.  Then it follows that this degree sequence is not graphic as,
$$\sum_{i=1}^{M_{b}-1}a_{i}=M_{a}(M_{b}-1) > \sum_{i=1}^{M_{b}-1}\#(b_{k}\geq i) = M_{a}(M_{b}-2) + (M_{a}-1) = M_{a}(M_{b}-1) - 1.$$}
\end{proof}
With a subtle but natural observation we can generalize Theorem 3 to the case where we prohibit loops and the bound will be remarkably similar.

\begin{thm} Consider a bidegree sequence $\vec{d}\in\mathbb{N}_0^{(n,2)}$ where $\sum{a_{i}}=\sum{b_{i}}=n\bar{c}$.  If $\max a_{i}\leq M_{a}$ and $\max b_{i}\leq M_{b}$, where $(M_{a}+1)M_{b}\leq n\bar{c}$, then $\vec{d}$ is graphic.  In particular, if $\max \vec{d} = M_{a}=M_{b} \leq \sqrt{\frac{1}{4}+n\bar{c}}-\frac{1}{2}$, then $\vec{d}$ is graphic.
\end{thm}  
\begin{proof}

 First, we show that for $j \leq M_b$, the $j$th inequality from the 
 Gale-Ryser Theorem holds.   We have $\sum_{i=1}^{j}a_{i}\leq jM_{a}$ and 
 \[ jM_{a}=j(M_{a}+1)-j\leq_{*}\sum_{i=1}^{n}\min(b_{i},j)-j \leq \sum_{i=1}^{j}\min(b_{i},j-1)+\sum_{i=j+1}^{n}\min(b_{i},j).
 \]
   The starred inequality follows from applying Lemma 1 to minimize the sum $\newline \sum_{i=1}^{n}\min(b_{i},j)$ with respect to the constraint that max$(b_{i})\leq M_{b}$, where 
   \[  \sum_{i=1}^{n}\min(b_{i},j)=\sum_{i=1}^{j}\sum_{k=1}^{n}\bold{1}_{(b_{k}\geq i)}=\sum_{i=1}^{j}\#(b_{z}:b_{z}\geq i, 1\leq z\leq n),
   \]
     $M_{b}(M_{a}+1) \leq n\bar{c}$, and $\sum_{i=1}^{n}b_{i}=n\bar{c}$.  For the minimizing sequence thus obtained, $b_{1}^*=...=b_{M_{a}+1}^*=M_{b}$, as $M_{b}(M_{a}+1)\leq n\bar{c}$, and hence 
     \[ \sum_{i=1}^{n} \min (b_i^*,j) = j(M_a+1).
     \]

For $j\geq M_{b}+1$, we can eliminate the minimum functions, as now $j-1\geq M_{b}\geq b_{i}$ for all $i$.  Thus, 
\[
\sum_{i=1}^{j}\min(b_{i},j-1)+\sum_{i=j+1}^{n}\min(b_{i},j)=\sum_{i=1}^{n}b_{i},
\]
 and $\sum_{i=1}^{n}b_{i}\geq\sum_{i=1}^{j}a_{i}$ for $j\leq n$ as the $a_{i}$'s are nonnegative and by assumption $\sum_{i=1}^{n}a_{i}=\sum_{i=1}^{n}b_{i}$.

In the special case where $M_{a}=M_{b}$, it follows that $M_{a}=\left\lfloor\sqrt{\frac{1}{4}+n\bar{c}}-\frac{1}{2}\right\rfloor$, as this quantity is the largest integer that satisfies the inequality  $M(M+1)\leq n\bar{c}$. $\newline$ 
\end{proof}

For large graphs, Theorems 3 and 4 provide \textcolor{black}{bounds that ensure graphicality of a bidegree sequence while allowing for a relatively large maximal degree.}  However, for many graphs we also have information about a lower bound on in- and out-degrees.  Consequently, we aim to prove two types of extensions.  In one extension, Theorem 5, we assume that there is a nonzero minimum degree, which in turn enables us to construct a more flexible sufficient condition on the maximum degree to guarantee graphicality.  The other type of extension, given in Corollary 5, also exploits the working assumption of a minimum degree in order to allow a small set of exceptional degrees to exceed the bound  on the maximum proposed in Theorem 3 while maintaining graphicality.  To simplify the proof Theorem 5, we prove the following corollary, which has utility of its own in verifying graphicality of a degree sequence.  

Henceforth, we drop the notation $M_{a}$ and $M_{b}$ and \textcolor{black}{refer to the maximum value of the bidegree sequence as $M$, given by}
\[
 \textcolor{black}{ M = \max_i \{ \max_i a_i, \max_i b_i \}.}
 \] 

\begin{cor}{\color{black}Suppose that a bidegree sequence $\vec{d}\in\mathbb{N}_0^{(n,2)}$ has a maximum value $M<n$ and for the associated in-degree sequence, $\#(a_{i}=M)=k$, where $M\leq k$. Then $\vec{d}$ is graphic with loops.  More generally, if for some $k\in\mathbb{N}$,  both $M\leq k$ and $Mk\leq n\bar{c}$ hold, then $\vec{d}$ is graphic with loops.}
\end{cor}
\begin{proof}
{\color{black}It follows by assumption that $$M^{2}\leq Mk \leq n\bar{c}$$ and hence Theorem 3 applies.}
\end{proof}

An application of Corollary 2 provides us with a powerful check for graphicality with loops.  Indeed, suppose we have verified the first $k$ inequalities of the Gale-Ryser theorem where the maximum is large ($M>>\sqrt{n\bar{c}}$).  We can then look at the residual degree sequence where the residual maximum is much more friendly and construct a linear upper bound for the remaining inequalities based on the new maximum of the residual degree sequence to verify whether the remaining $n-k$ inequalities hold.

Before we move on to prove Theorem 5, we \textcolor{black}{make an adjustment to Corollary 2 to handle} graphicality without loops.

\begin{cor} Suppose that a bidegree sequence $\vec{d}\in\mathbb{N}_0^{(n,2)}$ has a maximum value $M<n$ and for the associated in-degree sequence,  $\#(a_{i}=M)=k$, where $M< k$. Then $\vec{d}$ is graphic.  More generally, if there exist $k,M\in\mathbb{N}$, such that $M<k$ and $Mk\leq n\bar{c}$, then $\vec{d}$ is graphic.
\end{cor}
\begin{proof}
 {\color{black}
 The proof is analogous to the prior corollary, since the assumptions give $$M(M+1)\leq Mk \leq n\bar{c}$$ and application of Theorem 4 completes the proof.
 
 }
\end{proof}

We are now ready to prove the following sufficient condition on graphicality, and later we show that it is an asymptotically sharp refinement over the condition proven by Zverovich and Zverovich \cite{Z92}.
\begin{thm} Consider a bidegree sequence $\vec{d}=(\vec{a},\vec{b})\in\mathbb{N}_0^{(n,2)}$ where $\sum{a_{i}}=\sum{b_{i}}=n\bar{c}$ and $\min\vec{d}=m \textcolor{black}{\in [1..n]}$.  Define 
\begin{equation}
\label{eq:kstar}
k_{*}=m +\sqrt{m^{2}+n(\bar{c}-2m)}
\end{equation}
 and let $k=\left\lceil k_*\right\rceil$ if \textcolor{black}{$k_*$ is real} and $k=1$ otherwise. If 
 \begin{equation}
 \label{eq:M}
 M:=\max \vec{d}\leq \min(\left\lfloor n\frac{\bar{c}-m}{k}+m\right\rfloor,n),
 \end{equation}
   then $\vec{d}$ is graphic with loops.
\end{thm}
\begin{proof}

\textcolor{black}{
To start, suppose that a bidegree sequence has maximal degree $M$ given by (\ref{eq:M}) with $k_*, k$ as defined in the statement of the theorem.  
Note that  $n\bar{c} - M \leq Mk + (n-(k+1))m \leq n\bar{c} - m$.
We can thus consider the vector $\vec{b}^{*}$ such that $b^{*}_{1}=b^{*}_{2}=...=b^{*}_{k}=M$, $b^{*}_{k+1}=r$  and all other $b^{*}_{i}=m$, where we choose the remainder $r$ such that $kM+r+(n-(k+1))m = n\bar{c}$ and thus $m \leq r \leq M$.
Recall that $\bold{F}(j,\vec{b}) = \sum_{i=1}^n \min (b_i,j)$ and that we have an alternative representation of $\bold{F}(j,\vec{b})$ from (\ref{eq:Fdef}).
Since $\bold{F}(m,\vec{b})=nm$ for all $\vec{b}$ with minimum $m$, we can apply Lemma 1 to show that $\vec{b}^{*}$ is a minimizer of $\bold{F}$.}

At this stage, we \textcolor{black}{assume that $k_*$ is real, and we} would like to show that the first $k$ Gale-Ryser inequalities hold.
In fact, because of the nonzero minimum $m$, the first $m$ Gale-Ryser inequalities are trivially satisfied since $M\leq n$; in particular, in the special case of $m=n$, Theorem 5 is true.
So, the only case we need to consider here is the case when $m<k$ and $m<n$, which we henceforth assume.
As previously, $\sum_{i=1}^{j}a_{i}\leq jM=:V(j)$.  Since $V(j)$ is linear and $W(j):=\sum_{i=1}^n\min(b_{i},j)$ is concave, by Lemma 1, to verify the first $k$ inequalities of the Gale-Ryser Theorem, it suffices to show that $V(k)\leq W(k)$.
 
Therefore, 
we seek to verify that the $k$th inequality holds for our minimizing vector $\vec{b}^{*}$.

The definition of $r$ implies that the following two equivalent equations both hold:
 \begin{equation}\label{eq:max}kM+r-m+(n-k)m=n\bar{c}\iff kM+r=n\bar{c}-(n-k-1)m. 
 \end{equation}  
Using $r \geq m$ in (\ref{eq:max}), it follows that 
\begin{equation} \label{eq:ai}
\sum_{i=1}^{k}a_{i}\leq kM\leq n\bar{c}-(n-k)m.
\end{equation}
Furthermore, for $m<k$ and $m<n$, 
\begin{equation}
\label{eq:bi}
nm +k(k-m)\leq \sum_{i=1}^n \min(b_{i}^{*},k)\leq \sum_{i=1}^n\min(b_{i},k),
\end{equation} 
since the middle quantity is $k^2+\min(r,k)+(n-(k+1))m$ and $\min (r,k) \geq m$.

\textcolor{black}{Combining (\ref{eq:ai}) and (\ref{eq:bi}) implies that the first $k$ Gale-Ryser inequalities will be guaranteed to hold as long as $n\bar{c}-(n-k)m\leq nm +k(k-m)$ or, equivalently, as long as $$R(k)=k^{2}-2m k+2nm-n\bar{c}\geq 0.$$  
Note  that $R(k)\geq 0$ for $k\geq k_{*}$ where $k_{*}$ as defined in (\ref{eq:kstar}), which we have assumed for now to be real,  is the larger root of $R(k)$.
Unfortunately, $k_{*}$ does not have to be a natural number.  But for $k=\left\lceil k_{*} \right\rceil=k_{*}+z$, where $z:=k-k_{*}\in [0,1)$,     it follows that $R(k)\geq 0$. }
 
\textcolor{black}{We have now established that under this choice of $M$, the first $k$ Gale-Ryser inequalities hold, where $k$ may be equal to 1 or $k_*$ depending on whether or not $k_*$ is real.
We no longer assume that $k \neq 1$, and we next proceed to show that the remaining inequalities hold as well.}
Assume that we have a remainder $b^{*}_{k+1}=r>m$, and we wish to verify the $(k+1)$st inequality for our minimizing vector. 
 We will construct another polynomial, $S(\cdot)$, such that if the polynomial is nonnegative when evaluated at $(k+1)$, then the  $(k+1)$st inequality in the Gale-Ryser Theorem holds. 
 Furthermore we will show that for our choice of $k$, $S(u)\geq 0$ for $u\geq k$.

 It follows from equation (\ref{eq:max}) that $kM+r= n\bar{c}+m-(n-k)m$ and we would like to find \textcolor{black}{a condition on $k$ that ensures} that $n\bar{c}+m-(n-k)m \leq nm +k(k+1-m)+1\leq \bold{F}(k+1,\vec{b}^{*})$, where the $+1$ in the middle quantity is a lower bound on $r$.
We therefore define 
 $$S(k)=k^{2}+k(1-2m)+(2n-1)m-n\bar{c}+1,$$ with largest root
 $$k_{**}=m-\frac{1}{2}+\sqrt{m^{2}+n\bar{c}-2nm-\frac{3}{4}}$$
 (if the roots are positive), 
 and by an analogous argument to that used for $R(\cdot)$, it follows that for all $u\geq k_{**}$, $S(u)\geq 0$.
 By noting that $k\geq k_{*}>k_{**}$ \textcolor{black}{ (if $k_*$ is real)}, we have shown that 
 \[ \sum_{i=1}^{k+1}a_{i}\leq F(k+1,\vec{b}).
 \]

To finish off the proof, it remains to  verify the $\{k+2,k+3,....,n\}$ inequalities.  Define $\delta=1$ if $r>m$ and 0 otherwise.
Since  we have shown that $\sum_{i=1}^{k+\delta}a_{i}\leq \sum_{i=1}^{k+\delta}\min(b^{*}_{i},k+\delta)\leq \bold{F}(k+\delta,\vec{b})$ and we know that
$\sum_{i=1}^{n}a_{i}=\sum_{i=1}^{n}\min(b^{*}_{i},n)=\bold{F}(n,\vec{b})$, Lemma 1 guarantees that for all $j$ such that $k+\delta\leq j\leq n$,
$\sum_{i=1}^{j}a_{i}\leq \sum_{i=1}^{n}\min(b^{*}_{i},j)\leq \bold{F}(j,\vec{b})$.  Thus, the proof is complete. 
\end{proof}

Now we state the analogous result to show that a degree sequence is graphic.  We also provide a sketch of the proof, which follows similarly to the proof of Theorem 5,  and leave it to the reader to fill in the missing details.

\begin{thm} Consider a bidegree sequence $\vec{d}=(\vec{a},\vec{b})\in\mathbb{N}_0^{(n,2)}$ where $\sum{a_{i}}=\sum{b_{i}}=n\bar{c}$ and $\min\vec{d}=m$, where $m\leq n-1$.  Let $k_{*}=m+1+\sqrt{(m+1)^{2}+n(\bar{c}-2m)}$ and define $k=\left\lceil k_*\right\rceil$ if $k_*$ is real and $k=1$ otherwise.  If $$ \max \vec{d}\leq \min(\left\lfloor n\frac{\bar{c}-m}{k}+m\right\rfloor,n-1),$$ then $\vec{d}$ is graphic.  \end{thm}
\begin{proof}  
Analogously to the proofs of Theorem 5 and Corollary 3, to construct the desired sufficient condition on $M$, we want the following inequalities to hold:
$$\sum_{i=1}^{n}a_{i}\leq (M+1)j\leq \sum_{i=1}^{n}\min(b_{i},j).$$
Applying Lemma 1, it suffices to consider the case when $j=k$, where $\#(a_{i}=M)\leq k$.
However, we know that for the remainder $r$ in our usual minimizer construction, as given in equation (\ref{eq:max}), $r = n\bar{c}-Mk - m(n-k-1)$ or equivalently $kM+r-m+(n-k)m=n\bar{c}$, and consequently, 
$$kM+k\leq n\bar{c}-(n-k)m+k.$$
Additionally we know that $nm +k(k-m)\leq \sum_{i}\min(b_{i}^{*},k)\leq \sum_{i}\min(b_{i},k)$, where $\vec{b}^{*}$ is the same as in the proof of Theorem 5.  
We construct the polynomials $$R_{*}(k)=R(k)-k=k^{2}-2k(m+1)+2nm-n\bar{c}$$ and $$S_{*}(k)=S(k)-k-1=k^{2}-2m(k)+(2n-1)m-n\bar{c}.$$

Let $k_{*}$ and $k_{**}$ be the larger of the two roots of $R_{*}(k)$ and $S_{*}(k)$, respectively:

$$k_{*}=m+1+\sqrt{m^{2}+2m+1+n\bar{c}-2nm},$$

$$k_{**}=m+\sqrt{m^{2}+n\bar{c}+m-2nm}.$$

It follows that if 
$k>k_{*}$, then both $R_{*}(k)$ and $S_{*}(k)$ are nonnegative.  
As before define $k=\left\lceil k_{*}\right\rceil$.
Consequently, since $k(M-m)-nm\leq n\bar{c}$, we get the constraint that $M\leq \left\lfloor\frac{n\bar{c}-(n-k_{*})m}{k}\right\rfloor$.  This verifies the first $k$ or, if there is a remainder, $k+1$ inequalities of the Gale-Ryser Theorem.  As in the end of the proof of Theorem 5, invoking Lemma 1 will verify the remaining inequalities.
\end{proof}
Although the maximum value in Theorem 5 is easy to compute, it is not obvious if this bound is superior to both Theorem 3 (where $M_{a}=M_{b}$) and Theorem 2.  Therefore we provide the following proof of superiority.
\begin{cor}
Consider degree sequences with a fixed minimum degree $m$, fixed average degree $\bar{c}$ such that $\bar{c}>m$, where we allow the number of nodes, $n$, in the sequence to vary.  Notationally, for each  $J\in\{2,3,4,5,6\}$, we can define $H_J(n,m,\bar{c})$ such that each Theorem J shows that a bidegree sequence is graphic (with loops) if the maximum degree $M\leq H_{J}(n,m,\bar{c})$.  Then $\lim_{n\rightarrow\infty} \frac{H_{q}(n,m,\bar{c})}{H_{p}(n,m,\bar{c})}\geq 1$, for each $q\in \{5,6\}$ and $p \in \{2,3,4\} $.
\end{cor}
\begin{proof}
We only prove the result for $H_{5}(n,m,\bar{c})$, although the proof for $H_{6}(n,m,\bar{c})$ is identical.
We break the analysis up into two cases.

\medskip

\noindent {\bf Case 1:} $\bar{c}-2m\leq 0$

\medskip

In this case, $k\leq 2m$, and our condition on the maximum is $O(n)$, which is far superior to $O(\sqrt{n})$.  
\medskip

\noindent {\bf Case 2:} $\bar{c}-2m>0$

\medskip

Since we are only interested in asymptotic analysis, it suffices to consider the case when \textcolor{black}{$k_{*}\in\mathbb{N}$ (so $k=k_{*}$)} and $H_{5}(n,m,\bar{c})=n\frac{\bar{c}-m}{k}+m\in\mathbb{N}$. 
Consequently, $$H_{5}(n,m,\bar{c})=n(\sqrt{m^{2}+n(\bar{c}-2m)}-m)\frac{\bar{c}-m}{n(\bar{c}-2m)}+m$$ 
$$=(\sqrt{m^{2}+n(\bar{c}-2m)}-m)\frac{\bar{c}-m}{(\bar{c}-2m)}+m$$

$$=\sqrt{n\frac{(\bar{c}-m)^{2}}{\bar{c}-2m}+m^{2}(\frac{\bar{c}-m}{\bar{c}-2m})^{2}}-m
(\frac{m}{\bar{c}-2m}).$$

Note that asymptotically for large $n$, fixed $m$ and $\bar{c}$, if $p=2$, 
then 
\[
\lim_{n\rightarrow\infty}\frac{H_{p}}{\sqrt{4mn}}=1,
\]
 while if $p\in\{3,4\}$, 
then 
\[
\lim_{n\rightarrow\infty}\frac{H_{p}}{\sqrt{\bar{c}n}}\leq1.
\]
So asymptotically, to demonstrate that Theorem 5 is indeed more powerful than Theorems 2-4, we want to show that $\frac{(\bar{c}-m)^{2}}{\bar{c}-2m}\geq\bar{c}$ and $\frac{(\bar{c}-m)^{2}}{\bar{c}-2m}\geq 4m$.  For the ensuing discussion, let $\bar{c}=x$, $m=y$, with $x>2y$ by assumption.

First consider $\frac{(x-y)^{2}}{x-2y}\geq x \iff x^{2}-2xy+y^{2}\geq x^{2}-2yx$, which is true always.
 Next, note that $\frac{(x-y)^{2}}{x-2y}\geq 4y\iff x^{2}-2xy+y^{2}\geq 4xy-8y^{2}\iff x^{2}\geq 6xy-9y^{2}$.  Since $x>0$, this inequality is equivalent to $1\geq 6(\frac{y}{x})-9(\frac{y}{x})^{2}$.  Using another change of variables, where $a=\frac{y}{x}$, we want to know when $1\geq 6a-9a^{2}$. Taking the derivative of the right hand side implies that the maximum value of the right hand side occurs at $a=\frac{1}{3}$.  Since $6(\frac{1}{3})-9(\frac{1}{9})=1$, we conclude that asymptotically, Theorem 5 is more powerful than Theorems 2-4.
\end{proof}

As a simple example, note that if $m=1$, $\bar{c}=4$, $n=10$, and $M=6$ then $k=6$ and
$M \leq \lfloor n(\bar{c}-m)/k + m \rfloor = 6$, so Theorem 5 holds, but $(m+M)^2/4 = \frac{49}{4}>12 > 10=mn$ so Theorem 2 fails. 

As a final comment regarding Theorem 5, while it is not surprising that we can sharpen the bounds on the maximum by including an additional parameter $(n\bar{c})$, corresponding to the total number of edges, it is not readily apparent why the bound would dramatically change from $O(\sqrt{n})$ to $O(n)$ as {\it two} times the minimum number of edges of a node approaches the average number of edges. 

We now conclude our results section with a corollary of Theorem 5 that yields a more flexible graphicality criterion in which the degrees of some nodes can exceed the upper bound mentioned in Theorem 5.
\begin{cor}
Consider a bidegree sequence $\vec{d}=(\vec{a},\vec{b})\in\mathbb{N}_0^{(n,2)}$ where $\sum_{i=1}^n{a_{i}}=\sum_{i=1}^n{b_{i}}=n\bar{c}$ and $\min\vec{d}=m$, with $m \leq n$ and $\max \vec{d}\leq n$.  Without loss of generality, take the $a_{i}$ to be arranged in non-increasing order. Assume that there exists an $R$ such that $\sum_{i=1}^{R}a_{i}= n\lambda$, and $\sum_{i=1}^{R}b_{i}\leq n\lambda$ where $\lambda < m$ and $n-n\frac{\lambda}{m}-R\geq 1$. 
Next, define $M=\max_{i\geq R}\max(a_{i},b_{i})$ and $k_{*}=m +\sqrt{m^{2}+n(\bar{c}-2m)+Rm}$.  Let $k=\left\lceil k_{*}\right\rceil$ if $k_{*}$ is real and $k=1$ otherwise. If $ M\leq \min(\left\lfloor \frac{n\bar{c}-nm-n\lambda+Rm}{k}+m\right\rfloor,n)$ and if either $k\leq M$ or $k\leq n-n\frac{\lambda}{m}-R$,  then $\vec{d}$ is graphic with loops.
\end{cor}
\begin{proof}
The proof is quite similar to that of Theorem 5, so we only provide a sketch and leave the details to the reader.  Given that we defined  $\sum_{i=1}^{R}a_{i} = n\lambda$, and $\lambda<m$, the first $R$ inequalities of the Gale-Ryser Theorem are trivially satisfied.  Furthermore, the first $m$ inequalities are satisfied trivially as well since $\max \vec{d} \leq n$.

As in the proof of Theorem 5, we note that for arbitrary $k>0$, $\sum_{i=1}^{k+R}a_{i}\leq kM+n\lambda$.
For our minimizing degree sequence, $n\bar{c}=kM+(r-k)+(n-k)m+n\lambda-Rm$, \textcolor{black}{where $r$ is defined in the proof of Theorem 5} and hence $kM\leq n\bar{c}-(n-k)m+Rm-n\lambda$ since $r \geq m$. 
Thus, $kM+n\lambda 
\leq n\bar{c}+Rm-(n-k)m$.

Similarly, since we can assume that $k>m$, it follows that $nm+k(k-m)\leq\sum_{i=1}^{n}\min(b_{i},k)\leq \sum_{i=1}^{n}\min(b_{i},k+R)$, provided that $k\leq M$.  

Putting the bounds on $\sum_{i=1}^{k+R}a_{i}$ and $\sum_{i=1}^{n}\min(b_{i},k+R)$ together, to satisfy the remainder of the first $k+R$ Gale-Ryser inequalities, under the assumption that $k\leq M$, we want to fulfill the inequality $nm+k(k-m)-n\bar{c}-Rm+(n-k)m\geq 0$, where equality is achieved when $k=m+\sqrt{m^{2}+n(\bar{c}-2m)+Rm}$. Consequently, if $M\leq \left\lfloor \frac{n\bar{c}-n\lambda-(n-R)m}{\left\lceil k\right\rceil}+m\right\rfloor,$ then the first $k+R\leq M$ inequalities in the Gale-Ryser Theorem will be satisfied.
To finish off the proof, we then consider the case where $k>M$.
We know that $\sum_{i=1}^{n}\min(b_{i},k+R)\geq n\bar{c}-n\lambda$ and $\sum_{i=1}^{k+R}a_{i}\leq n\bar{c}-(n-k-R)m$.  Consequently, we require that $n\lambda+Rm\leq (n-k)m$, or equivalently, $k\leq (n-n\frac{\lambda}{m}-R)$.  Hence, our assumptions imply that  the first $M$ inequalities hold.

Now suppose for simplicity that $r=0$.  For the degree sequence that maximizes the in-degree vector $\vec{a}$ in the Gale-Ryser Theorem, $a_{j}=m$ for all $j>k+R$, and hence  $\sum_{i=1}^{j}a_{i}$ grows linearly in $j$ for these $j$.  We can therefore complete the proof by invoking  Lemma 1, since $\sum_{i=1}^{n}\min(b_{i},k)$ is concave in $k$.  This result implies that the remaining inequalities must hold.  
In the case where $r>0$, as before in Theorems 5 and 6, we exploit the existence of the remainder to construct refined inequalities that demonstrate that our prior choice for $M$ is indeed correct.
\end{proof}

Recalling Counterexample 1, the only way we were able to construct a degree sequence that was not graphic was by having many nodes with degrees greater than $\sqrt{n\bar{c}}$.   In contrast, Corollary 5 tells us that in an asymptotic sense, as long as we have a relatively small number of nodes $R$ with degrees that surpass $O(\sqrt{n})$, such that the sum of their degrees is $n\lambda=O(n^{1-\tau})$ for some $\tau>0$, then asymptotically we still have graphicality provided that $O(n)$ nodes are bounded in-degree by essentially the same bound derived in Theorems 5 and 6.  This observation is useful, for example, for broadening the graphicality criteria for so-called scale free networks with exponent greater than $2$. For such networks, we find that the expected number of edges contributed by nodes of degree greater than $\sqrt{n}$ is  $n\int_{\sqrt{n}}^{n}\frac{x}{x^{2+\tau}}dx=O(n^{1-\frac{\tau}{2}})$.  In this setting, Corollary 5 can be viewed in parallel with the prior work of Chen and Olvera-Cravioto \cite{Chen13}, who proved that provided the sum of the in-degrees equals the sum of the out-degrees, randomly generated degree sequences from a scale-free distribution with a finite mean (that is, with an exponent greater than 2) are asymptotically (almost surely) graphic. 

\section{Discussion}
While the famous Gale-Ryser inequalities (e.g., \cite{Berger11, Miller13}) provide necessary and sufficient conditions for a degree sequence to be graphic, checking these inequalities and using them to generate graphs \cite{Kim12} can be computationally inefficient.  Work by Zverovich, Alon, Cairns \textcolor{black}{and their collaborators} provides simplified sufficient conditions for graphicality; however, these conditions assume that the in-degree vector equals the out-degree vector for directed graphs and are posed in terms of the minimum and maximum of the degree sequence.  In our analysis, we drop the assumption that the in-degree vector equals the out-degree vector and prove an alternative sufficient condition for graphicality incorporating the average degree (Theorems 3, 5).  We prove that for fixed minimum and average degree, for sufficiently large $n$, Theorem 5 provides more flexible conditions to demonstrate graphicality than those provided by prior work.  The proof method used in this paper builds heavily on that used by Dahl and Flatberg \cite{dahl} and Miller \cite{Miller13} in their approaches to relaxing the graphicality conditions in the Erd\"{o}s-Gallai and Gale-Ryser Theorems, with the key idea being to exploit the discrete concavity of the functions appearing in the relevant inequalities.    Note that while all results in this paper are stated in terms of bidegree sequences for directed graphs, \textcolor{black}{these results apply immediately to bipartite graphs, while} the proof methods will extend directly to the case of undirected graphs as well.

In Counterexample 1, we show that we cannot expect to do much better than our sufficient conditions for graphicality using bounds on the average degree alone.  However, we also notice that  to construct a degree sequence that is not graphic, we must choose many nodes to have large degree.   
This observation motivates Corollary 5, which says that as long as only a relatively small number of node degrees exceed $O(\sqrt{n})$, we still have graphicality.  Interpreted in an asymptotic sense, we can relate this result to the work of Chen and Olvera-Cravioto \cite{Chen13}, which shows that asymptotically, degree sequences generated from scale-free distributions with exponent greater than $2$ almost surely will be graphic. 

\section{Acknowledgements}
This work was partially supported by NSF Award DMS 1312508.
The authors thank Jeffrey Wheeler (Pitt) for a thorough reading of a draft of this manuscript and for many suggestions that helped improve the exposition of this work, Noga Alon (Tel Aviv) for his insightful comments and the anonymous reviewers at SIDMA for their helpful comments.

\end{document}